\documentclass[11pt,a4paper]{article}%
\usepackage[centertags]{amsmath}
\usepackage{amsfonts}
\usepackage{amssymb}
\usepackage{amsthm}
\usepackage{epsfig}
\usepackage{setspace}
\usepackage{ae}
\usepackage{eucal}
\usepackage[usenames]{color}%
\usepackage{graphicx}

\theoremstyle{plain}
\newtheorem{thm}{Theorem}[section]
\newtheorem{lem}[thm]{Lemma}

\theoremstyle{definition}

\theoremstyle{remark}

\newtheorem{rem}[thm]{Remark}

\renewcommand{\div}{\operatorname{div}}

\begin{document}

\title{ On the multivariate Burgers equation and the incompressible Navier-Stokes equation (Part I) }
\author{J\"org Kampen }
\maketitle

\begin{abstract} 
We provide a constructive  global existence proof for the multivariate viscous Burgers equation system defined on the whole space or on a domain isomorphic to the $n$-torus and with time horizon up to infinity and $C^{\infty}$- data (satisfying some growth conditions if the problem is posed on the whole space). The proof is by a time discretized semi-explicit perturbative expansion in transformed coordinates where the convergence is guaranteed by certain a priori estimates. The scheme is useful in order to define computation schemes for related equation systems of fluid dynamics.
\end{abstract}

\footnotetext[1]{Weierstrass Institute for Applied Analysis and Stochastics,
Mohrenstr.39, 10117 Berlin, Germany. Support by DFG Matheon is ackknowledged.
\texttt{{kampen@wias-berlin.de}}.}

\section{Introduction}
In order to understand the dynamics of fluids we need to understand equations of the form
\begin{equation}\label{qparasyst1}
\left\lbrace \begin{array}{ll}
\frac{\partial u_i}{\partial t}=\nu\sum_{j=1}^n \frac{\partial^2 u_i}{\partial x_j^2} 
-\sum_{j=1}^n u_j\frac{\partial u_i}{\partial x_j},\\
\\
\mathbf{u}(0,.)=\mathbf{h},
\end{array}\right.
\end{equation}
 where $\nu$ is some strictly positive constant (i.e. $\nu>0$) and $1\leq i\leq n$. Global existence for such equations can be obtained by standard methods, e.g. by an upper bound estimate for the solution with respect to the maximum norm (cf. the estimate (\ref{aprioriest}) below). In this paper we look for a constructive global solution scheme for $\mathbf{u}=(u_1,\cdots,u_n)$ on some domain $[0,\infty)\times \Omega$, where $[0,\infty)$ is some time interval and $\Omega \subseteq \mathbb{R}^n$, and where some initial data $\mathbf{h}=(h_1,\cdots,h_n)^T$ with $h_i\in C^ {\infty}\left(\Omega\right) $ are given. Special interest is either in the case $\Omega={\mathbb R}^n$ or in the case of periodic boundary conditions, i.e. solutions of the form $u_i:[0,\infty)\times {\mathbb R}^n\rightarrow {\mathbb R},~1\leq i\leq n$ which satisfy the condition $u_i(t,x)=u_i(t,x+\gamma)$ for $\gamma\in \Gamma$ with some lattice $\Gamma$, say $\Gamma={\mathbb Z}^n$. The latter periodic functions define solutions on the quotient space $[0,\infty)\times {\mathbb R}^n/{\mathbb Z}^n$ in a natural way, and the spatial part of the latter is isomorphic to the $n$-torus ${\mathbb T}^n$. Alternatively, one may look at the restriction of periodic functions of the kind above to a reference domain $\Omega=\Omega_0$, where 
\begin{equation}
\Omega_0=\left\lbrace x\in {\mathbb R}^n|0\leq x_j< 1,~~1\leq j\leq n\right\rbrace,
\end{equation}
and then identify according to 
\begin{equation}\label{per}
u_i(t,x)=u_i(t,x+\gamma)~~\mbox{ if } ~~x,x+\gamma\in \partial\Omega_0~~\mbox{ and }\gamma\in \Gamma,
\end{equation}
where, $\partial \Omega_0$ denotes the boundary of $\Omega_0$. We shall  speak loosely of the Cauchy problem on the $n$-torus $\Omega ={\mathbb T}^n$. Clearly, the data $\mathbf{h}=(h_1,\cdots ,h_n)^T$ are assumed to be periodic if $\Omega={\mathbb T}^n$. Our method can be adapted to other domains $\Omega\subset {\mathbb R}^n$ with additional boundary conditions of Dirichlet, von Neumann or mixed type. Furthermore, in prinicpal our method can be applied to fluids on compact Riemannian manifolds $M$, where  an equation analogous to (\ref{qparasyst1}) is of the form
\begin{equation}
\frac{\partial {\bf u}}{\partial t}-{\cal L}{\bf u}+\nabla_{\bf u}{\bf u}=0.
\end{equation}
Here, $\nabla$ denotes the covariant derivative, and ${\cal L}$ may be the Hodge Laplacian 
$\Delta=-(dd^*+dd^*)$ or the Bochner Laplacian 
$-\nabla^*\nabla$ with 
$\nabla:C^{\infty}\left(M,TM \right)\rightarrow C^{\infty}\left(M,T^*\otimes T \right)$, or some similar operator arising from the covariant derivative.

However, in order to avoid too much technicalities  we mainly stick to the case $\Omega={\mathbb R}^n$, and confine ourselves to cursory remarks regarding Cauchy problems on $\Omega={\mathbb T}^n$, on manifolds or initial-boundary  value problems. Our goal in this paper is to define a constructive scheme for the multivariate Burgers equation which also can be interpreted in a probabilistic way. In a subsequent paper we shall show that such schemes may be useful in order to define global computation schemes for the incompressible Navier-Stokes equation.

Back to the multivariate Burgers equation, if we put a potential source term on the right side of (\ref{qparasyst1}), then we have a generalization of Burgers one-dimensional model for fluids, i.e.
\begin{equation}\label{multburg}
\frac{\partial\mathbf{u}}{\partial t}-\nu \Delta \mathbf{u}+ (\mathbf{u} \cdot \nabla) \mathbf{u} =-\nabla F(t,x). 
\end{equation}
\begin{rem}
If the velocity field equals the gradient of some potential function $\phi$ (in analogy to a physical potential of a conservative force) for initial time $t_0$, then we have
\begin{equation}\label{pot}
\mathbf{u}(t_0,x)=-\nabla \phi(t_0,x)~~\mbox{($t_0$ some initial time)},
\end{equation}
and equation (\ref{multburg}) is equivalent to the evolution of $\phi$ with

\begin{equation}\label{pot1}
\frac{\partial \phi}{\partial t}-\nu \Delta \phi=\frac{1}{2}|\phi|^2+F(t,x).
\end{equation} 
So-called Hopf-Cole transformation may be used to transform to
\begin{equation}\label{cole}
\frac{\partial\alpha}{\partial t}-\nu \Delta \alpha=\frac{1}{2\nu}F\alpha.
\end{equation} 
Global existence and effcient computation schemes for such scalar equations are well known and for some conditions on $F$ explicit local expansions of the fundamental solution lead to schmes which are of particular efficiency (cf. \cite{Kale}). We know that data $\mathbf{u}(0,.)$ satisfy (\ref{pot}) for some potential function $\phi$ if and only if the curl of $\mathbf{u}$ is zero. Incompressibility is then given for free since $\nabla \cdot \left(  \nabla\times\mathbf{u}\right) =\div\left( \nabla\times \mathbf{u}\right) =0$ for smooth data.
However, for studying fluids we need to analyze the situation for data more general than just potential data.
We deal with this more general situation in this paper where no Hopf-Cole transformation is available.
\end{rem}
\begin{rem}
Note that global existence for the system (\ref{qparasyst1}) can be obtained from estimates of the solution in the supremum norm in terms of the supremum norm of the initial data, i.e. of the form
\begin{equation}\label{apriorimax}
\max_{j}\sup_{x\in \Omega}|u_j(t,x)|\leq \max_{j}\sup_{x\in \Omega}|f_j(x)|,
\end{equation}
which may be obtained form estimates of the form
\begin{equation}\label{aprioriest}
\frac{\partial}{\partial t}\|u(t,.)\|_{H^s}\leq \|u(t,.)\|_{H^{s+1}}\sum_{i,j}\sum_{|\alpha|+|\beta|\leq s}\|D^{\alpha}u_iD^{\beta}u_j\|_{L^2}-2\|\nabla u\|^2_{H^s}
\end{equation}
for some postive $s\in {\mathbb R}$.
\end{rem}

Systems of form (\ref{qparasyst1}) are building blocks of well-known models in fluid mechanics and finance. 
Let us consider the Navier-Stokes equations
\begin{equation}\label{nav}
    \frac{\partial\mathbf{v}}{\partial t}-\nu \Delta \mathbf{v}+ (\mathbf{v} \cdot \nabla) \mathbf{v} = - \nabla p + \mathbf{f}~~~~t\geq 0,~~x\in\Omega,
\end{equation}
together with the incompressibility assumption
\begin{equation}\label{navdiv}
\nabla \cdot \mathbf{v} = 0,~~~~t\geq 0,~~x\in\Omega, 
\end{equation}
and together with the initial conditions
\begin{equation}\label{navinit}
\mathbf{v}(0,x)=\mathbf{h}(x),~~x\in\Omega,
\end{equation}
where $\mathbf{h}(x)=\left(h_1(x),\cdots ,h_n(x) \right) $ is a given function with components in $C^ {\infty}\left(\Omega \right)$.
This equation is to be solved for an unknown velocity vector $\mathbf{v}(t,x)=(v_1(t,x),\cdots,v_n(t,x))$ with $(t,x)\in [0,\infty)\times \Omega$, and
$\mathbf{v}(t,x)\in\left[C^\infty([0,\infty)\times\Omega)\right]^n$, and with a scalar pressure $p(t,x)\in C^\infty([0,\infty)\times\Omega^n)$, where $\nu=\frac{\eta}{\rho}>0$ is a the strictly positive viscosity constant. Furthermore, $$\mathbf{f}(t,x)=\left(f_1(t,x),\cdots ,f_n(t,x) \right) $$ is a given external force where $f_i\in C^{\infty}\left([0,\infty)\times \Omega\right) $ for each $1\leq i\leq n$. 

\begin{rem}
Let us mention here, that in case of a compact Riemannian manifold the Navier-Stokes equation takes the form
\begin{equation}\label{navmani}
\frac{\partial {\bf u}}{\partial t}-{\cal L}{\bf u}+\nabla_{{\bf u}}{\bf u}=-\mbox{grad}~p,~~\mbox{div}~{\bf u}=0,~~{\bf u}(0,.)={\bf h}(.),
\end{equation}
where we use the notation 'grad' instead of the 'nabla'-operator notation in this context in order to avoid confusion with the symbol of covariant derivative. On $\mbox{ker}~\mbox{div}$ one has
\begin{equation}
{\cal L}{\bf u}=\Delta {\bf u}+2\mbox{Ric}~{\bf u},
\end{equation}
where $\mbox{Ric}$ denotes the Ricci tensor. For manifolds with constant curvature, i.e. $\mbox{Ric}= c\cdot \mbox{id}$, we get ${\cal L}=\Delta +2c\cdot \mbox{id}$, and in this case we can eliminate the pressure and write (\ref{navmani}) in the form
\begin{equation}\label{leray}
\frac{\partial {\bf u}}{\partial t}-{\cal L}{\bf u}+P\nabla_{{\bf u}}{\bf u}=0,~~{\bf u}(0,.)={\bf h}(.),
\end{equation}
where ${\bf h}(.)$ is still assumed to have divergence zero. Here, $P$ is the Leray projection, i.e. the orthogonal projection of $L^2\left(M,TM\right)$ onto the kernel of the divergence operator.
\end{rem}

If $\Omega={\mathbb R}^n$, then a natural growth condition for the data is that for all $x\in\mathbb{R}^n$
\begin{equation}\label{growthh}
    \vert \partial^\alpha \mathbf{h}(x)\vert\le \frac{C_{\alpha K}}{(1+\vert x\vert)^K}\qquad 
\end{equation} 
holds for any multiindices $\alpha$ and $K$ with some constants $C_{\alpha K}$. 
Furthermore, in this case it is assumed that
\begin{equation}\label{growthf}
    \vert \partial^m_t\partial^\alpha \mathbf{f}(t, x)\vert\le \frac{C_{\alpha m K}}{(1+\vert x\vert + t)^K}\qquad ~~\mbox{for all}~~ \qquad (t,x)\in [0,\infty)\times \mathbb{R}^n,
\end{equation}
and for any multiindices $\alpha$ and and integers $K$ with nonnegative integers $m$, and with some constants $C_{\alpha m K}$.
\begin{rem}
In the following we assume that $\mathbf{f}=0$, although the method considered is not restricted to this case.
Note that the condition (\ref{growthh}) ensures that
\begin{equation}
\mathbf{h}\in \left[ H^s\left({\mathbb R}^n\right)\right]^n,
\end{equation}
where for each $s\in{\mathbb R}$ $H^s$ denotes a Sobolev space, i.e.
\begin{equation}
H^s=H^s\left({\mathbb R}^n\right)=\left\lbrace f\in {\cal S}'\left({\mathbb R}^n\right) |\hat{f}~~\mbox{is a function and}~~\|f\|_s^2<\infty\right\rbrace ,
\end{equation}
along with
\begin{equation}
 \|f\|_s^2\equiv\int_{{\mathbb R}^n}|\hat{f}(y)|^2(1+|y|^2)^sdy<\infty,
\end{equation}
and where $\hat{f}$ denotes the Fourier transform of $f$, and 
${\cal S}'\left({\mathbb R}^n\right) $ denotes a space of tempered distributions.
We shall use these function spaces for the convergence of the semi-explicit perturbative expansion on the whole space $\left[0,\infty \right)\times {\mathbb R}^n$. In case $p=2$ the definition is equivalent to the more classical definition of spaces $H^{k,p}$ for nonnegative integers $k$ and real $p\in [1,\infty)$, i.e.
\begin{equation}
H^{k,p}\left({\mathbb R}^n\right) :=\left\lbrace u\in L^p\left({\mathbb R}^n\right)|D^{\alpha}u\in L^p\left({\mathbb R}^n\right)  \mbox{ for $|\alpha|\leq k$}\right\rbrace . 
\end{equation}
More generally, for $s\in{\mathbb R}$ one defines
\begin{equation}
H^{s,p}\left({\mathbb R}^n\right):=\Lambda^{-s}L^p\left({\mathbb R}^n\right),~~\mbox{where}~~\Lambda^su={\cal F}^{-1}\left( (1+y^2)^{s/2}\hat{u}\right),   
\end{equation}
and ${\cal F}^{-1}$ denotes the inverse of the Fourier transform operator.
\end{rem}
Next for an arbitrary step size $T>0$ in the original time coordinates we shall find a time discretization $\left\lbrace  T_l,~l=0,\cdots ,N\right\rbrace $ with $T_l=\sum_{k=1}^l\rho_k\uparrow \infty $ where we define a local scheme recursively for each time interval $\left[ T_{l-1},T_l\right]$ in such a way that the growth of the solution $u$ is well-controlled in order to allow for a global scheme. We shall see that for our global scheme the step size is of the form
\begin{equation}
\rho_l\sim\frac{C}{l}
\end{equation}
for a fixed constant $C$ such that the scheme is indeed global with respect to time. 
We noticed above that a growth condition is available from a priori estimates, but we shall get an estimate independently. One advantage of our approach is that our consideration can be generalized in order to define global schemes for Navier-Stokes type equations.  A second advantage is that we get a constructive solution scheme which leads to computation schemes. In order to prove local convergence with respect to time we shall consider a simple time transformation at each time step, i.e., we shall consider transformations of the form  $t=\rho_l \tau$  where  $t\in [T_{l-1},T_l]$ , and where $\tau \in [(l-1)T,lT]$ for natural numbers $l\geq 1$. On each such interval we denote $k$th approximation of the solution by
\begin{equation}
(\tau,x)\rightarrow {\bf u}^{\rho,k,l}=(u_1^{\rho, k,l},\cdots ,u_n^{\rho,k,l})
\end{equation}
where the index $k$ in ${\bf u}^{\rho,k,l}$ refers to iterations at a given time step and $l-1$ is the number of time steps accomplished.
We start with the first time step solving (\ref{qparasyst1}) on $[0,\rho_1 T]\times {\mathbb R}^n$ with respect to $t$ coordinates and on $[0,T]$ with respect to $\tau$ coordinates, where $t=\rho_1 \tau$ (the parameter $\rho_1$ will be determined below).
Let the function $(\tau,x)\rightarrow {\bf u}^{\rho,1,0}=(u_1^{\rho, 1,0},\cdots ,u_n^{\rho,1,0})$ be the solution of the equation system
\begin{equation}\label{scalparasystlin10}
\left\lbrace \begin{array}{ll}
\frac{\partial u^{\rho,1,0}_i}{\partial \tau}=\rho_1\left( \nu\sum_{j=1}^n \frac{\partial^2 u^{\rho,1,0}_i}{\partial x_j^2} 
-\sum_{j=1}^n h_j\frac{\partial u^{\rho,1,0}_i}{\partial x_j}\right) ,~~1\leq i\leq n,\\
\\
{\bf u}^{\rho,1,0}(0,.)={\bf h}.
\end{array}\right.
\end{equation} 
Since the $h_i,~1\leq i\leq n$ are known, at this stage we have no coupling, and (\ref{scalparasystlin10}) are essentially $n$ scalar equations. 
Having defined ${\bf u}^{\rho,k,0}=(u_1^{\rho, k,0},\cdots,u_n^{\rho, k,0})^T$ with ${\bf u}^{\rho,k,0}(0,x)={\bf h}$ for some $k\geq 1$ we define ${\bf u}^{\rho,k+1,0}$ recursively to be the solution of (corresponding to the domain $[0,\rho_1 T]\times {\mathbb R}^n$ in original coordinates $t$ and to $[0,\rho_1 T]\times {\mathbb R}^n$ in transformed coordinates $\tau$.)
\begin{equation}\label{urhok}
\left\lbrace \begin{array}{ll}
\frac{\partial u^{\rho,k+1,0}_i}{\partial \tau}=\rho_1\left( \sum_{j=1}^n \frac{\partial^2 u^{\rho,k+1,0}_i}{\partial x_j^2} 
-\sum_{j=1}^n u^{\rho,k,0}_j\frac{\partial u^{\rho,k+1,0}_i}{\partial x_j}\right) ,~~1\leq i\leq n,\\
\\
\mathbf{u}^{\rho, k+1,0}(0,.)=\mathbf{h}.
\end{array}\right.
\end{equation} 
Then subtracting the equation (\ref{urhok}) for ${\bf u}^{\rho,k,0}$ from the equation (\ref{urhok}) for ${\bf u}^{\rho,k+1,0}$ we get 
\begin{equation}\label{deltaurhok0}
\left\lbrace \begin{array}{ll}
\frac{\partial \delta u^{\rho,k+1,0}_i}{\partial \tau}=\rho_1\left( \sum_{j=1}^n \frac{\partial^2 \delta u^{\rho,k+1,0}_i}{\partial x_j^2} 
-\sum_{j=1}^n u^{\rho,k,0}_j\frac{\partial \delta u^{\rho,k+1,0}_i}{\partial x_j}\right) \\
\\
\hspace{2cm}-\rho_1\sum_j\left(\delta u^{\rho,k,0}_j\frac{\partial u^{\rho,k,0}}{\partial x_j} \right),\\ 
\\
\mathbf{\delta u}^{\rho,k+1,0}=0,
\end{array}\right.
\end{equation}
where $\delta u^{\rho,k+1,0}_{i}=u^{\rho,k+1,0}_{i}-u^{\rho,k,0}_{i}$ for $k\geq 1$, and 
$\delta u^{\rho,1,0}_j= u^{\rho,1,0}_j-h_j$.
Now look at the family of equations in (\ref{deltaurhok0}). They are linear equations with first order coefficients $u^{\rho,k,0}_j$ available externally from the previous step respectively. Again these are essentially $n$ scalar equations. The coupling is 'externalized' to the coefficients $u^{\rho,k,0}_j$ and the source terms carrying information from the previous step.
Next $\rho_1>0$ can be chosen such that on the domain $[0,\rho_1 T]\times {\mathbb R}^n$
\begin{equation}\label{funci}
u^{\rho ,0}_i=u^{\rho,1,0}_i+\sum_k \delta u^{\rho, k+1,0}_i, 1\leq i\leq n
\end{equation}
is a strongly convergent series and defines a function ${\bf u}^{\rho ,0}$ in a suitable classical space.
Having defined the scheme for time steps $0,\cdots ,l-1$ we define the scheme with time step index $l$ (i.e. the scheme for the solution on the domain $[T_{l-1},T_l]$ in original time-coordinates $t$) in an analogous way (using the semi-group property).
Let the function $(\tau,x)\rightarrow {\bf u}^{\rho,1,l}=(u_1^{\rho, 1,l},\cdots ,u_n^{\rho,1,l})$ be the solution of the equation system
\begin{equation}\label{scalparasystlin1*}
\left\lbrace \begin{array}{ll}
\frac{\partial u^{\rho,1,l}_i}{\partial \tau}=\rho_l\left( \sum_{j=1}^n \frac{\partial^2 u^{\rho,1,l}_i}{\partial x_j^2} 
-\sum_{j=1}^n u^{\rho,1,l-1}_j\frac{\partial u^{\rho,1,l}_i}{\partial x_j}\right) ,~~1\leq i\leq n,\\
\\
{\bf u}^{\rho,1,l}((l-1)T,.)={\bf u}^{\rho,l-1}((l-1)T,.),
\end{array}\right.
\end{equation} 
and let $(\tau,x)\rightarrow {\bf u}^{\rho,k+1,l}=(u_1^{\rho, k+1,l},\cdots ,u_n^{\rho,k+1,l})$ be the solution of the equation system
\begin{equation}\label{scalparasystlin1*2}
\left\lbrace \begin{array}{ll}
\frac{\partial u^{\rho,k+1,l}_i}{\partial \tau}=\rho_l\left( \sum_{j=1}^n \frac{\partial^2 u^{\rho,k+1,l}_i}{\partial x_j^2} 
-\sum_{j=1}^n u^{\rho,k,l}_j\frac{\partial u^{\rho,k+1,l}_i}{\partial x_j}\right) ,~~1\leq i\leq n,\\
\\
\mathbf{u}^{\rho, k+1,l}((l-1)T,.)={\bf u}^{\rho,l-1}((l-1)T,.).
\end{array}\right.
\end{equation}
\begin{rem}
Note that in the original coordinates the initial condition at the $l$th time step is
$${\bf u}^{1,l}((T_{l-1},.)={\bf u}^{\rho,1,l-1}((l-1)T,.).$$
\end{rem}

Then subtracting the equation (\ref{scalparasystlin1*2}) for ${\bf u}^{\rho,k,l}$ from the equation (\ref{scalparasystlin1*2}) for ${\bf u}^{\rho,k+1,l}$ we get 
\begin{equation}\label{deltaurhok}
\left\lbrace \begin{array}{ll}
\frac{\partial \delta u^{\rho,k+1,l}_i}{\partial \tau}=\rho_l\left( \sum_{j=1}^n \frac{\partial^2 \delta u^{\rho,k+1,l}_i}{\partial x_j^2} 
-\sum_{j=1}^n u^{\rho,k,l}_j\frac{\partial \delta u^{\rho,k+1,l}_i}{\partial x_j}\right) \\
\\
\hspace{2cm}-\rho_l\sum_j\left(\delta u^{\rho,k,l}_j\frac{\partial u^{\rho,k,l}_i}{\partial x_j} \right),\\ 
\\
\mathbf{\delta u}^{\rho,k+1,l}=0.
\end{array}\right.
\end{equation}
This leads to a functional series
\begin{equation}\label{funci2}
u^{\rho ,l }_i=u^{\rho,1,l}_i+\sum_k \delta u^{\rho, k+1,l}_i, 1\leq i\leq n
\end{equation}
which is defined on $[(l-1)T,lT]\times {\mathbb R}^n$ in $\tau$-coordinates corresponding to $[T_{l-1},T_l]\times {\mathbb R}^n$ in original $t$-coordinates. If we can choose
$\rho_l,~l\geq 1$ such that 
\begin{equation}
\sum_{l\geq 1} \rho_l\uparrow \infty \mbox{ as } l\uparrow \infty,
\end{equation}
then we have a global scheme. We shall prove in this paper that such a choice is possible such that the scheme is indeed global.
Moreover, we shall see that the convergence of (\ref{funci}) and (\ref{funci2})  is such that differentiation can be done term by term, and convergence of the differentiated functional series is uniform and absolute in a pointwise sense. Plugging in the series (\ref{funci}), (\ref{funci2}) into (\ref{qparasyst1}) and rearranging the summands
 according to the laws one checks easily that the limit in 
(\ref{funci}, \ref{funci2}) leads indeed to a solution of the multivariate Burgers equation.

The construction outlined has an interesting consequence: the solution of the multivariate Burgers equation has a representation in terms of a series of fundamental solutions of (scalar !) linear equations. This may be useful for numerical purposes since there are efficient schemes for fundamental solutions of scalar linear parabolic equations in terms of WKB-expansions or related analytic expansions.
Note also that this leads to probabilistic schemes where elaborated weighted Monte-Carlo schemes may be used (cf. \cite{KKS} and \cite{FK} and references therein). This would be attractive especially in cases of complicated boundary conditions, or if coupling with other equations (e.g. Maxwell equations) lead to complcated models where traditional techniques are of very limited success so far. However, the following semi-explicit formulas may also be used in the context of sparse grids or adaptive sparse grids etc. (regularity for the solution  makes this option attractive). Let us spell out the scheme in terms of fundamental solutions (densities).
For each positive integer $k\geq 1$ and given $u^{\rho, l}_k,~1\leq k\leq n$ and time step $l$ let us denote the fundamental solution of the linear equation
\begin{equation}\label{funddeltaurhok}
\begin{array}{ll}
\frac{\partial w_i}{\partial \tau}=
\rho_l\left( \sum_{j=1}^n \frac{\partial^2 w_i}{\partial x_j^2}+\sum_{j=1}^n u^{\rho,k,l}_j\frac{\partial w_i}{\partial x_j}\right) 
\end{array}
\end{equation}
by $\Gamma^{\rho,\Omega ,l}_k$ (we can construct this solution by the classical Levy expansions and find very effcient higher order approximations in terms of local analytic expansions; the upperscript $\Omega$ of $\Gamma^{\rho,\Omega ,l}_k$ indicates that we apply such approximations on bounded domains obtained either by transformations or by cutoff).
Then given $u_i^{\rho,k,l}$ we have the representation
\begin{equation}\label{funddeltarep}
\begin{array}{ll}
\delta u^{\rho, k+1,l}_i(\tau ,x)=\\
\\
\int_0^{\tau}\int_{\Omega}\rho\sum_j\delta u_j^{\rho,k,l}(\sigma ,y)\frac{\partial u^{\rho,k,l}_i}{\partial x_j}(\sigma ,y)\Gamma^{\rho,\Omega ,l}_k(\tau ,x,\sigma,y)d\sigma dy.
\end{array}
\end{equation}
The series starts at $l=0$ with the fundamental solution $\Gamma^{\rho,\Omega ,0}_1$ of 
\begin{equation}\label{funddeltaurho1}
\begin{array}{ll}
\frac{\partial w_i}{\partial \tau}=
\rho\left( \sum_{j=1}^n \frac{\partial^2 w_i}{\partial x_j^2}+\sum_{j=1}^n h_j\frac{\partial w_i}{\partial x_j}\right), 
\end{array}
\end{equation}
and with
\begin{equation}
\delta u^{\rho ,1}_j(\tau,x)=u^{\rho, 1}_j-h_j=\int_{\Omega} h_j(y)\Gamma^{\rho,\Omega ,0}_1(\tau,x;0,y)dy-h_j(x),
\end{equation}
and proceeds analogously for $l>0$.
We can make the representation (\ref{funddeltarep}) at each time step $l$ more explicit, i.e.
\begin{equation}
\begin{array}{ll}
\delta u^{\rho, k+1,l}_i(\tau ,x)=\\
\\
\int_0^{\tau}\int_{\Omega}\cdots \int_0^{\sigma_n}\int_{\Omega}\left( \sum_{j_1=1}^n \rho_l\delta u_{j_1}^{\rho,1,l}\frac{\partial u^{\rho,1,l}_i}{\partial x_{j_1}}\right)(\sigma_1 ,y_1) \\
\\
\Pi_{m=2}^k\left( \sum_{j_m=1}^n\rho_l u^{\rho,m,l}_{j_m}\frac{\partial u^{\rho,m,l}_i}{\partial x_{j_m}}\right) (\sigma_m ,y_m)\times \\
\\
\Pi_{m=2}^{k-1}\Gamma^{\rho,\Omega ,l}_m(\sigma_{m},y_m ,\sigma_{m-1},y_{m-1})\Gamma^{\rho,\Omega ,l}_k(\tau ,x,\sigma_k,y_k)d\sigma_1 dy_1\cdots d\sigma_k dy_k.
\end{array}
\end{equation}
Hence we have the formal representation for the solution ${\bf u}=(u_1,\cdots ,u_n)^T$ of the multivariate Burgers equation with zero source term
\begin{equation}
\begin{array}{ll}
u^{\rho,1,0}_i+\sum_{l=0}^{\infty}\left( \sum_{k=0}^{\infty}\delta u^{\rho, k+1,l}_i\right) (\tau ,x)=\\
\\
u^{\rho,1,0}_i+\sum_{l=0}^{\infty}\int_0^{\tau}\int_{\Omega}\cdots \int_0^{\sigma_n}\int_{\Omega}\left( \sum_{j_1=1}^n \rho_l\delta u_{j_1}^{\rho,1,l}\frac{\partial u^{\rho,1,l}_i}{\partial x_{j_1}}\right)(\sigma_1 ,y_1) \\
\\
\sum_{k=0}^{\infty}\Pi_{m=2}^k\left( \sum_{j_m=1}^n\rho_l u^{\rho,m,l}_{j_m}\frac{\partial u^{\rho,m,l}_i}{\partial x_{j_m}}\right) (\sigma_m ,y_m)\times \\
\\
\Pi_{m=2}^{k-1}\Gamma^{\rho,\Omega ,l}_m(\sigma_{m},y_m ,\sigma_{m-1},y_{m-1})\Gamma^{\rho,\Omega ,l}_k(\tau ,x,\sigma_k,y_k)d\sigma_1 dy_1\cdots d\sigma_k dy_k,
\end{array}
\end{equation}
where for each $l$ the whole related series for $k$ has to be summed up, and where
\begin{equation}\label{urec}
\begin{array}{ll}
u^{\rho, 1,0}_i(\tau,x)=\int_{\Omega} h_j(y)\Gamma^{\rho,\Omega ,l}_1(\tau,x;0,y)dy,~~1\leq i\leq n\\
\\
\frac{\partial}{\partial x_k}u^{\rho, 1,0}_i(\tau,x)=\int_{\Omega} h_j(y)\frac{\partial}{\partial x_k}\Gamma^{\rho,\Omega ,0}_1(\tau,x;0,y)dy,~~1\leq i,k\leq n\\
\\
\delta u^{\rho, 1,0}_j(\tau,x)=u^{\rho, 1,0}_j(\tau,x)-h_j(x),~~1\leq j\leq n.
%
\end{array}
\end{equation}
Analogously for time steps $l>0$.
The latter recursive scheme leads to semi-explicit recursive formulas for the multivariate Burgers equation. Note that the fundamental solutions of the parabolic recursively linear  equations involved have an explicit representation. It leads also to new higher order numerical schemes.

In the next section we show how a series $(\rho_l)$ can be chosen such that we have global convergence of the scheme above.

\section{Convergence of the global scheme}
We may assume that the time step size in transformed coordinates equals $1$, i.e. w.l.o.g. we assume $T=1$ (the scheme of section 1 is a global scheme if it converges to the solution, where $\sum_{l}\rho_l\uparrow \infty$). This leads to a time scheme with time-step size of the $l$th time step proportional to $\frac{1}{l}$.
\begin{rem}
Justification of the definite article in 'the solution', i.e. uniqueness, is provided below in Theorem 3.1.
\end{rem}

We construct the solution $u(t,x)=u^{\rho}(\tau ,x)$ of the multivariate Burgers equation in terms of a time discretized scheme of functional series where for all $l\geq 1$ and for $(t,x)\in [T_{l-1},T_l]\times {\mathbb R}^n$ resp. $(\tau,x)\in [(l-1),l]\times {\mathbb R}^n$ we have
\begin{equation}\label{funci3}
u^{\rho}_i(\tau,x):=u^{\rho ,l }_i(\tau,x)=u^{\rho,1,l}_i(\tau,x)+\sum_k \delta u^{\rho, k+1,l}_i(\tau,x), 1\leq i\leq n.
\end{equation}
We show that there is a sequence of numbers $(\rho_k)$ and a related time discretization
$\left\lbrace T_l|l=1,\cdots ,N\right\rbrace $ where $T_N\uparrow \infty$ as $N\uparrow \infty$ such that ${\bf u}^{\rho}=(u^{\rho}_1,\cdots ,u^{\rho}_n)^T$ defined via (\ref{funci3}) satisfies the multivariate Burgers equation.

\begin{rem}
It is interesting that the condition of decay for the initial data in (\ref{growthh}) implies that we can transform the Cauchy problem to a compact domain considering the transformation $y_i=\arctan(x_i),~1\leq i\leq n$ and the related equation for the function ${\bf \tilde{u}}$, where
\begin{equation}
\tilde{u}_i(t,y)=u_i(t,x)~~\mbox{resp.}~{\tilde{u}}_i^{\rho}(\tau,y)=u_i^{\rho}(t,x).
\end{equation}
Then we have 
\begin{equation}
\frac{\partial u_i}{\partial x_j}=\frac{\partial {\tilde {u}}_i}{\partial y_j}\frac{1}{1+\tan^2(y_j)},
\end{equation}
and
\begin{equation}
\frac{\partial^2 u_i}{\partial x_j\partial x_k}=\frac{\partial^2 {\tilde{u}}_i}{\partial y_j^2}\frac{1}{\left( 1+\tan^2(y_j)\right)^2}-\delta_{jk}\frac{\partial {\tilde {u}}_i}{\partial y_j}\frac{2\tan(y_k)}{\left( 1+\tan^2(y_j)\right)\left( 1+\tan^2(y_k)\right)},
\end{equation}
where $\delta_{ij}$ denotes the Kronecker delta. The Cauchy problem for ${\bf \tilde{u}}=(\tilde{u}_1,\cdots \tilde{u}_n)^T$ related to the multivariate Burgers equation live on the domain $[0,T]\times \left( -\pi/2,\pi/2\right)^n $.
 It is interesting that, for example, the relation 
\begin{equation}\label{ubehavinfty}
\left( 1+|x|^2\right)^3\frac{\partial^2 u_i}{\partial x_j\partial x_k}\downarrow 0 \mbox{ as } |x|\uparrow \infty
\end{equation}
enforces 
\begin{equation}
\frac{\partial^2 {\tilde{u}}_i}{\partial x_j\partial x_k}\downarrow 0 \mbox{ as } |x|\uparrow \pi/2,
\end{equation}
and our analysis below shows that (\ref{ubehavinfty}) is indeed the case if the initial data satisfying (\ref{growthh}). This makes Dirichlet data natural for the related initial-boundary value problem for ${\bf \tilde{u}}$. We shall use this in the second part of this paper when we deal with the Navier Stokes equation.
\end{rem}
In order to show the global convergence of the functional series $u_i$ we use classical results for scalar equations. First we shall consider convergence with respect to  the supremum norm $|.|_0$, where for bounded mesurable functions we define
\begin{equation}
|g|_0:=\sup_{(t,x)\in [0,T]\times{\mathbb R}^n}|g(t,x)|.
\end{equation}
As we shall see the growth of the solution ${\bf u}$ of our scheme is controlled by the maximum principle. This growth is linear with respect to the time horizon and controls the time step size of the scheme (which has reciprocal value essentially). In order to ensure the convergence of the substeps $k$ of the $l$th time step  (which lead to a local solution of the multivariate Burgers equation with respect to time) we observe that we can control the first order derivatives of the correction terms $\delta u^{\rho,k,l}_i$ uniformly. The time step size is chosen in such a way that we get a uniform and absolute estimates with respect to the norm $|.|_{1,2}$. At each time step we also establish an uniform and abolute convergent geometric series bound  for the correction terms series $\left( \delta u^{\rho,k,l}_i\right)_k$ for each $1\leq i\leq n$ with respect to the norm $|.|_{1,2}$. This implies that derivatives can be taken term by term and that the functional series (\ref{funci3}) is indeed a solution of the multivariate Burgers equation. So much for the outline.

Next we need some elementary technical preparations (quite standard). Define the Euclidean distance in ${\mathbb R}^{n+1}$ between the points $y_1=(t_1,x_1), y_2=(t_2,x_2)$ by
\begin{equation}
e(z_1,z_2)=\sqrt{|t_1-t_2|}+|x_1-x_2|.
\end{equation}
Although we do not consider convergence w.r.t. H\"{o}lder norms we shall use them at one step of the argument.  
If $w$ is a function in a domain $D\subset{\mathbb R}^{n+1}$ we denote
\begin{equation}
[w]_{\delta/2,\delta, D}=\sup_{y_1\neq y_2;y_1,y_2\in D}\frac{|w(y_1)-w(y_2)|}{e^{\delta}(y_1,y_2)}.
\end{equation}
Next define
\begin{equation}
|w|_{\delta/1,\delta;D}=|w|_{0 ,D}+[w]_{\delta/2,\delta;D}.,
\end{equation}
Furthermore we shall use some classical spaces
\begin{equation}\label{classnorm0}
|w|_{0,1;D}:=|w|_{0;D}+\sum_{i=1}^n|w_{x_i}|_{0;D},~\mbox{and}
\end{equation}
\begin{equation}\label{classnorm}
|w|_{1,2;D}:=|w|_{0;D}+\sum_{i=1}^n|w_{x_i}|_{0;D}+|w_t|_{0;D}+\sum_{i,j=1}^n|w_{x_ix_j}|_{0;D}.
\end{equation}
Note that the latter norms do not define Banach spaces. However, we shall use the fact that a functional series which is uniformly and absolutely bounded with uniformly and absolutely bounded derivatives can be differentiated term by term. For this matter (\ref{classnorm}) is useful. 
In the following for $T\in (0,\infty)$ let $D=Q_T=[0,T]\times {\mathbb R}^n$. 
For vector-valued functions ${\mathbf w}=(w_1,\cdots ,w_n)^T$ we define
\begin{equation}\label{norm}
|{\mathbf w}|_{1,2,D}=\max_{1\leq i\leq n}|w_i|_{1,2,D}
\end{equation}
etc. Note that we can rewrite the Cauchy problem (\ref{qparasyst1}) in terms of coordinates $\tau(t)=\rho_lt$ for $\tau\in [T_{l-1},T_l]$ as in (\ref{basic*}) where we define
\begin{equation}
\rho(\tau):=\rho_l~\mbox{ if }~\tau\in [T_{l-1},T_l]
\end{equation}
for all $l\geq 0$.
Well, on key idea of the global scheme of this article is that the time step size decreases with order $\frac{C}{l}$ at the $l$th time step where the estimate of the solution is bounded by $C_l=C_0+l$ with $C_0$ a bound for the initial data (with respect to some strong Banach norm). This implies that the first order coefficients at the iteration substeps of each time step are uniformly bounded by some constant $C^*$ which is independent of the time step number $l$.

\begin{rem}
Note that a global (non-constructive) estimate (as (\ref{apriorimax}) based on (\ref{aprioriest}) above in remark 1.2) may lead us to a global scheme with an uniform time grid. However, we do not consider this for two reasons. First we want an explicit control of the time step in terms of the data and this is not provided by the a priori estimate (the size of the time steps may be uniform but so small that they are not useful in practice). Second, the present scheme may be extended to Navier-Stokes equations and there are no such a priori estimates available for the Navier-Stokes equation in the crucial case of dimension $3$.   
\end{rem}
 
We have
 \begin{thm}\label{burgglob}
 Assume that the initial data ${\bf h}$ satisfy the condition (\ref{growthh}). Let
 \begin{equation}
 C_0:=|{\bf u}^{\rho,1,0}|_{0}:=\max_{j\in \left\lbrace 1,\cdots n\right\rbrace} | u^{\rho,1,0}_j|_{0},
 \end{equation}
and define $C^*_n$ as in lemma \ref{lemma} below. 
 Define positive real numbers $(\rho_l)$ and $(C_l)$ recursively via
 \begin{equation}\label{rec}
 \begin{array}{ll}
 \rho_l=\frac{1}{ 4C^*_n C_{l-1}}=\frac{1}{ 4C^*_n\left(  C_{0}+(l-1)\right) }
\end{array}
\end{equation}
where $C_l=C_{l-1}+1$ for $l\geq 1$. For this sequence $(\rho_k)$ the functional series scheme above converges to the global classical solution ${\bf u^{\rho}}=(u^ {\rho}_1,\cdots ,u^{\rho}_n)^T\in \left[ C^{1,2}\left(\left[0,\infty\right)  \times \Omega  \right)\right]^n$ of
\begin{equation}\label{basic*}
\left\lbrace \begin{array}{ll}
\frac{\partial u^{\rho}_i}{\partial \tau}=\rho (\tau) \left( \sum_{j=1}^n \frac{\partial^2 u^{\rho}_i}{\partial x_j^2} 
+\sum_{j=1}^n u^{\rho}_j\frac{\partial u^{\rho}_i}{\partial x_j}\right) ,\\
\\
\mathbf{v}=\mathbf{h},
\end{array}\right.
\end{equation}
where $(t,x)\rightarrow {\bf u}(t,x)={\bf u^{\rho}}(\tau ,x)$ solves the original Cauchy problem (\ref{qparasyst1}) on $[0,\infty )\times {\mathbb R}^n$. The $C_l$ control the growth of the solution with respect to time.
\end{thm}

\begin{proof}
Recall that $T=1$. First we start with the scheme of the functional series (\ref{funci}) and prove its strong convergence with respect to the norm (\ref{norm}) 
on the domain $[0,\rho_1 ]\times {\mathbb R}^n=[0,T_1]\times {\mathbb R}^n$ in original time coordinates (corresponding to the domain $[0,1]\times {\mathbb R}^n$ in transformed time coordinates (i.e., $\tau$). Note that $\rho(\tau)=\rho_1$ on this domain.

Since the $n$ equations (for $1\leq i\leq n$) in (\ref{scalparasystlin10}) are $n$ identical scalar equations and the initial data and first order coefficient functions $h_i$ are smooth and have bounded derivatives, classical results tell us that a unique smooth solution  in $C^{\infty}$, i.e., $(\tau,x)\rightarrow {\bf u}^{\rho,1,0}=(u_1^{\rho, 1,0},\cdots ,u_n^{\rho,1,0})$ of the linear parabolic equation (\ref{scalparasystlin10}) exists which is bounded with bounded derivatives where all the $u^{\rho,1,0}_i$ are equal entries. Hence, by the maximum principle   (cf. part (i) of lemma \ref{lemma} below) for $C_0:=\max_{i\in \left\lbrace 1,\cdots,n\right\rbrace} \sup_{(\tau,x)\in Q_1}|h_i|>0$ we have
\begin{equation}
\max_{i\in \left\lbrace 1,\cdots ,n\right\rbrace} |u^{\rho,1,0}_i|_{0}\leq C_0.
\end{equation}
Moreover, according to part (ii) and part (iii) of lemma \ref{lemma} there exist a constant $C^*>0$ such that 
\begin{equation}
\sup_{(\tau,x)\in Q_1}{\Big |}\frac{\partial}{\partial x_i}u^{\rho,1,0}_j(\tau,x){\Big |}_0\leq C^*C_0,
\end{equation}
\begin{equation}
\sup_{(\tau,x)\in Q_1}{\Big |}\frac{\partial}{\partial t}u^{1,0}_j(\tau,x){\Big |}_0\leq C^*C_0,
\end{equation}
(note the original $t$-variable here), and for all
\begin{equation}
\sup_{(\tau,x)\in Q_1}{\Big |}\frac{\partial^2}{\partial x_i\partial x_m}u^{\rho,1,0}_j(\tau,x){\Big |}_0\leq C^*C_0.
\end{equation}
Define $C^*_n=\left(2+n+n^2\right)C^*C^*C^*$ for convenience. Choose
\begin{equation}
\rho_1:=\frac{1}{4C_0C^*_n},
\end{equation}
and assume inductively that
\begin{equation}\label{indscal}
\max_{i\in \left\lbrace 1,\cdots ,n\right\rbrace} {\Big |}\delta u^{\rho,k,0}_i{\Big |}_{0}\leq \frac{1}{2^k}|{\bf \delta u}^{\rho,1,0}|_0
\end{equation}
\begin{equation}\label{indder}
{\big |}\delta {\bf u}^{\rho,k,0}{\big |}_{0,1}:=\max_{i,j\in \left\lbrace 1,\cdots ,n\right\rbrace} {\Big |}\frac{\partial}{\partial x_j}\delta u^{\rho,k,0}_i{\Big |}_{0}\leq \max_{1\leq j\leq n}\frac{1}{2^k}{\Big |}\frac{\partial}{\partial x_j}{\bf \delta u}^{\rho,1,0}{\Big |}_0
\end{equation}
First we verfiy that the relations (\ref{indscal}) and (\ref{indder}) are inherited at step $k+1$. We have
\begin{equation}
\begin{array}{ll}
|{\bf \delta u}^{\rho, k+1,0}|_{0}\leq \rho_1|{\bf \delta u}^{\rho, k,0}|_{0}\left( C^* C_0+\sum_{m=2}^{k}{\Big |}\frac{\partial \delta {\bf u}^{\rho, m,0}}{\partial x_j}{\Big |}_{0}\right)\\
\\
\leq |{\bf \delta u}^{\rho, k,0}|_{0}\left( \rho_1C^* C_0+\rho_1 {\Big |}\frac{\partial \delta {\bf u}^{\rho, 1,0}}{\partial x_j}{\Big |}_{0}\right)\\
\\
\leq |{\bf \delta u}^{\rho, k,0}|_{0}\left( \frac{1}{4}+\rho_1 C^*C_0\right)\\
\\
\leq \frac{1}{2}|{\bf \delta u}^{\rho, k,0}|_{0}\leq \frac{1}{2^{k+1}}|{\bf \delta u}^{\rho, 1,0}|_{0}.
\end{array}
\end{equation}
In order to close the induction we  need estimates of the first order derivatives because these are involved in the source terms on the right side of the recursively defined equations. We shall use the fundamental solution. In the following we understand the term 'fundamental representation'  in the following sense: if $(t,x)\rightarrow f(t,x)$ is a function on a domain $[0,T]\times {\mathbb R}^n$ and $(t,x;s,y)\rightarrow p(t,x;s,y)$ is the fundamental solution of a parabolic equation then we call the function
\begin{equation}
(t,x)\rightarrow (f\ast p)(t,x):=\int_0^t\int_{{\mathbb R}^n}f(s,y)p(t,x;s,y)dsdy
\end{equation}
 a fundamental representation (we have classical convolution if $p$ is a normal density and is dependent on  $x-y$ only). The solution $\delta u^{\rho,2,0}_i$ (for $1\leq i\leq n$) and its first order spatial derivatives can be represented by fundamental representations involving the source term of (\ref{deltaurhok0}) and the fundamental solution of (\ref{scalparasystlin10}),  or the respective derivatives of this fundamental solution. Hence, according to lemma \ref{lemma} and remark \ref{remark*} below we get
\begin{equation}\label{est0}
\begin{array}{ll}
{\big |}\delta {\bf u}^{\rho,2,0}_i{\big |}_{1}=\max_{i,m\in \left\lbrace 1,\cdots ,n\right\rbrace} {\Big |}\frac{\partial}{\partial x_m}\delta u^{\rho,2,0}_i{\Big |}_{0}\\
\\
\leq \rho_1\max_{i\in \left\lbrace 1,\cdots ,n\right\rbrace}\sum_{j=1}^n {\Big |}\left( \delta u^{\rho,1,0}_j\frac{\partial  u^{\rho,1,0}_i}{\partial x_j}\right) \ast\frac{\partial}{\partial x_m}\Gamma^{\rho,1}_1{\Big |}_{0}\\
\\
\leq  \rho_1C_0nC^*C^*|{\bf \delta u}^{\rho,1,0}|_0\leq \rho_1C^*_n|{\bf \delta u}^{\rho,1,0}|_0
\end{array}
\end{equation}
where $\Gamma^{\rho,1}_1$  denotes the fundamental solution of (\ref{scalparasystlin10}) (which is identical for each $i$). Here we use lemma \ref{lemma} again. Note that we have $|{\bf \delta u}^{\rho,1,0}|_0$ on the right side of the latter estimate and not the derivative. Next we assume inductively that for $m\leq k$ we have
\begin{equation}
{\big |}\delta {\bf u}^{\rho,m,0}_i{\big |}_{1}\leq  \left( \rho_1C^*_nC_0\right)^{m-1}|{\bf \delta u}^{\rho,1,0}|_0.
\end{equation}
Then we get (a rough estimate is sufficient)
\begin{equation}
\begin{array}{ll}
|{\bf \delta u}^{\rho, k+1,0}|_{1}\leq \rho_1|{\bf \delta u}^{\rho, k,0}|_{0}\left( {\Big |}\frac{\partial \delta {\bf u}^{\rho, 1,0}}{\partial x_j}{\Big |}_0 +\sum_{m=2}^{k}{\Big |}\frac{\partial \delta {\bf u}^{\rho, m,0}}{\partial x_j}{\Big |}_{0}\right)\\
\\
\leq |{\bf \delta u}^{\rho, k,0}|_{0}\left( \rho_1C^* C_0+\rho_1 \left( \frac{\rho_1C^*_nC_0}{1-\rho_1C^*_nC_0}\right) \right)\\
\\
\leq |{\bf \delta u}^{\rho, k,0}|_{0}\left( \frac{1}{4}+\rho_1\cdot \frac{1}{3}\right)\\
\\
\leq \frac{1}{2}|{\bf \delta u}^{\rho, k,0}|_{0}\leq \frac{1}{2^{k+1}}|{\bf \delta u}^{\rho, 1,0}|_{0}.
\end{array}
\end{equation}

Note that
\begin{equation}
|\delta u^{\rho,1,0}_j|_0= |u^{\rho,1,0}_j-h_j|_0={\Big |}\int_0^{\rho_1}\sup_{s,x}\left( {\Big |}\frac{\partial}{\partial t}u^{1,0}_j(s,x){\Big |}\right) ds{\Big |}_0\leq C^*_nC_0\rho_1\leq \frac{1}{4}.
\end{equation}
Hence with the choice
\begin{equation}
\rho_1=\frac{1}{4C^*_nC_0}
\end{equation}
the functional series (\ref{funci}) is absolutely and uniformly convergent with respect to the norm $|.|_{0,1}$. Note that the relation (17) concerning the decay of the initial data at spatial infinity is preserved for the all functions $u^{\rho,k,0}_i$. Applying lemma \ref{lemma} once more by a similar argument we get the absolute and uniform convergence of the first time and the spatial derivatives of second order. Hence, the first order derivative with respect to time of the series (\ref{funci}) can be computed componentwise, i.e.
\begin{equation}\label{funcit}
\frac{\partial}{\partial t}u^{\rho ,0}_i=\frac{\partial}{\partial t}u^{\rho,1,0}_i+\sum_k \frac{\partial}{\partial t}\delta u^{\rho, k+1,0}_i, 1\leq i\leq n,
\end{equation}
and the same is true for the first and second order derivatives with respect to the spatial variables, i.e. for all $1\leq l\leq n$ we have
\begin{equation}\label{funcixl}
\frac{\partial}{\partial x_l}u^{\rho ,0}_i=\frac{\partial}{\partial x_l}u^{\rho,1,0}_i+\sum_k \frac{\partial}{\partial x_l}\delta u^{\rho, k+1,0}_i, 1\leq i\leq n,
\end{equation}
and for all $1\leq l,m\leq n$ we have
\begin{equation}\label{funcixl}
\frac{\partial^2}{\partial x_l\partial x_m}u^{\rho ,0}_i=\frac{\partial^2}{\partial x_l\partial x_m}u^{\rho,1,0}_i+\sum_k \frac{\partial^2}{\partial x_l\partial x_m}\delta u^{\rho, k+1,0}_i, 1\leq i\leq n.
\end{equation}
Using these relations we observe that the function ${\bf u}^{\rho,1}$ defined in terms of the series in (\ref{funci}) satisfies the multivariate Burgers equation in a classical pointwise sense. Hence, we have constructed a solution on $\left[0,T_1\right] \times{\mathbb R}^n$ in original time coordinates (reps. on $\left[0,T_1\right] \times{\mathbb R}^n$ in $\tau$- coordinates). On this domain we have
\begin{equation}\label{funci*}
|u^{\rho ,1 }_i|_{0}=|u^{\rho,1,0}_i+\sum_{k\geq 1} \delta u^{\rho, k+1,0}_i|_{0}\leq C_0+1,
\end{equation}
and
\begin{equation}\label{funci*}
|u^{\rho ,1 }_i|_{1,2}=|u^{\rho,1,0}_i+\sum_{k\geq 1} \delta u^{\rho, k+1,0}_i|_{1,2}\leq C^*_n\left( C_0+1\right) .
\end{equation}
Using this estimate we get a lower bound for the next time step. We define
\begin{equation}
\rho_2=\frac{1}{4C^*_nC_1}=\frac{1}{4C^*_n(C_0+1)}.
\end{equation}
As it turns out this choice controls the first order coefficient in the equations of the second time step such that $C^*_n$ is again a uniform estimate for the convolution with the fundamental solutions $\Gamma^{\rho,k}_1$. Let us observe this inductively and more closely.
Assume inductively that we have defined $C_{l-1}$ (up to some $l\geq 1$), and 
\begin{equation}\label{rhol}
\rho_l=\frac{1}{4C^*_n(C_{l-1}+1)}=\frac{1}{4C^*_n(C_{0}+l-1)},
\end{equation}
and that for $l-1\geq 1$ the function ${\bf u}^{\rho ,l-1 }$ has been constructed on the domain $[(l-2),(l-1)]\times {\mathbb R}^n$ (corresponding to the domain $[T_{l-2},T_{l-1}]\times {\mathbb R}^n$ in original time coordinates). We want to construct the solution ${\bf u}^{\rho ,l }$ on $[(l-1),l]\times {\mathbb R}^n$ (corresponding to the domain $[T_{l-1},T_{l}]\times {\mathbb R}^n$ in original coordinates) via the series
\begin{equation}\label{funcik}
u^{\rho ,l }_i=u^{\rho,1,l}_i+\sum_{k\geq 1} \delta u^{\rho, k+1,l}_i
\end{equation}
for $1\leq i\leq n$.

First, we consider the equation (\ref{scalparasystlin1*}) (recall that $T=1$). Similar as in the first time step we observe that the $n$ equations (for $1\leq i\leq n$) in (\ref{scalparasystlin1*}) are $n$ identical scalar equations and the initial data and first order coefficient functions $u^{\rho,1,l}_j$ are smooth and bounded. Classical results tell us that a unique smooth solution (i.e. in $C^{\infty}$) $(\tau,x)\rightarrow {\bf u}^{\rho,1,l}=(u_1^{\rho, 1,l},\cdots ,u_n^{\rho,1,l})$ exists which is bounded with bounded first order time derivative and bounded spatial derivatives up to second order, and which has equal entries. From the previous step we have an estimate of the first order coefficients in (\ref{scalparasystlin1*}), i.e.
\begin{equation}
\rho_l\max_{i\in \left\lbrace 1,\cdots ,n\right\rbrace} |u^{\rho,1,l-1}_i|_{1,2}\leq \rho_lC^*_nC_l\leq \frac{1}{4}.
\end{equation}
The choice of $\rho_l$ implies that we can choose the same constant $C^*_n$ as before (applying lemma \ref{lemma} as described in reamrk \ref{remark*} below) in order to get
\begin{equation}
\max_{i\in \left\lbrace 1,\cdots ,n\right\rbrace} |u^{1,l}_i|_{1,2}\leq C^*_nC_l.
\end{equation}
Applying lemma \ref{lemma} we get
\begin{equation}\label{est}
\begin{array}{ll}
 |{\bf \delta u}^{\rho,2,l}|_{1,2}\leq \rho_l\max_{i\in \left\lbrace 1,\cdots ,n\right\rbrace}\sum_{j=1}^n {\Big |}\delta u^{\rho,1,l}_j\frac{\partial  u^{\rho,1,l}_i}{\partial x_j}\ast\Gamma^l_1{\Big |}_{1,2}\\
\\
\leq \rho_lC^*_nC_l|{\bf \delta u}^{\rho,1,l}|_{0}
\end{array}
\end{equation}
where $\Gamma^{\rho,l}_1$  denotes the fundamental solution of (\ref{scalparasystlin1*}) (which is identical for each $i$) and the constant $C^*_n$ can be chosen the same as in the previous time steps (recall uniform upper bound for the first order coefficients implied by choice of time step size $\rho_l$). For the choice of $C^*_n$ we use again lemma \ref{lemma} as described in remark \ref{remark*}. 
We may assume w.l.o.g. that $C_l\geq 1$. Now assume inductively (we look at inductive substeps now) that for $2\leq m \leq k$ we have
\begin{equation}\label{induct}
\max_{i\in \left\lbrace 1,\cdots ,n\right\rbrace} |\delta u^{\rho,k,l}_i|_{1,2}
\leq  \left(\rho_lC_l C^*_{n}\right)^{m-1}|{\bf \delta u}^{\rho,1,l}|_{0}.
\end{equation}
Then we get
\begin{equation}
\begin{array}{ll}
|{\bf \delta u}^{\rho, k+1,l}|_{1,2}\leq \rho_l|{\bf \delta u}^{\rho, k,l}|_{0}\left( C^*_nC_l +\sum_{m=2}^{k}{\Big |}\frac{\partial \delta {\bf u}^{\rho, m,l}}{\partial x_j}{\Big |}_{0}\right)\\
\\
\leq |{\bf \delta u}^{\rho, k,0}|_{0}\left( \frac{1}{4}+\rho_1 C^*_nC_0\right)\\
\\
\leq \frac{1}{2}|{\bf \delta u}^{\rho, k,0}|_{0}\leq \frac{1}{2^{k+1}}|{\bf \delta u}^{\rho, 1,0}|_{0}.
\end{array}
\end{equation}
Here we observe inductively that $C^*_n$ is a uniform constant for estimates of the convolution with $\Gamma^{\rho, l}_k$ and its first order derivatives with respect to time and first and second order derivatives with respect to space.
%
Hence we have convergence of the functional series (\ref{funcik})  to the norm $|.|_{1,2}$. Again, uniform and absolute convergence of derivatives (first order with respect to time and up to second order with respect to space) with polynomial decay at spatial infinity leads to the conclusion that the functional series (\ref{funcik}) can be derived term by term. Hence, the first order derivative with respect to time of the series (\ref{funcik}) can be computed componentwise, i.e.
\begin{equation}\label{funcikt}
\frac{\partial}{\partial t}u^{\rho ,l}_i=\frac{\partial}{\partial t}u^{\rho,1,l}_i+\sum_k \frac{\partial}{\partial t}\delta u^{\rho, k+1,l}_i, 1\leq i\leq n,
\end{equation}
and the same is true for the first and second order derivatives with respect to the spatial variables, i.e. for all $1\leq l\leq n$ we have
\begin{equation}\label{funcixl}
\frac{\partial}{\partial x_l}u^{\rho ,l}_i=\frac{\partial}{\partial x_l}u^{\rho,1,l}_i+\sum_k \frac{\partial}{\partial x_l}\delta u^{\rho, k+1,l}_i, 1\leq i\leq n,
\end{equation}
and for all $1\leq l,m\leq n$ we have
\begin{equation}\label{funcixl}
\frac{\partial^2}{\partial x_l\partial x_m}u^{\rho ,l}_i=\frac{\partial^2}{\partial x_l\partial x_m}u^{\rho,1,l}_i+\sum_k \frac{\partial^2}{\partial x_l\partial x_m}\delta u^{\rho, k+1,l}_i, 1\leq i\leq n,
\end{equation}
Using these relations we observe that the function ${\bf u}^{\rho,l}$ defined in terms of the series in (\ref{funcik}) satisfies the multivariate Burgers equation in a classical poiintwise sense. Hence, we have constructed a solution on $\left[0,T_l\right] \times{\mathbb R}^n$. On this domain we have
\begin{equation}\label{funci*}
|u^{\rho ,l}_i|_{1,2}=|u^{\rho,1,l}_i+\sum_{k\geq 1} \delta u^{\rho, k+1,l}_i|_{1,2}\leq C_l+\sum_{k\geq 1}\frac{2\rho_lC^*_nC_l}{1-2\rho_lC^*_nC_l} \leq C_l+1 . 
\end{equation}
Using this estimate we get a lower bound for the next time step. We define
\begin{equation}
\rho_{l+1}=\frac{1}{4C^*_n(C_l+1)}
\end{equation}

Hence, the sequence $(\rho_l)$ satisfies
\begin{equation}
\rho_{l+1}\geq \frac{1}{4C^*(C_0+l)}
\end{equation}
such that $\sum_{l=1}^{N} \rho_l\uparrow \infty$ as $N\uparrow \infty$. This implies that the scheme is global in time.

\end{proof}

In the following we state the lemma used above in order to get the uniform constant $C^*_n$ which depends only on the dimension $n$ and the H\"{o}lder norm of the first order coefficients of the subproblems. This H\"{o}lder norm is uniformly bounded due to the choice of the 'time step size' $\rho_l$ (especially, it is independent of the substep $k$ and independent of the time step $l$). The independence of $l$ is in part due to the fact that (with the choice of $\rho_l$ and $C_l$ in the theorem) we have $\rho_l C_l\leq C$ for some constant $C$ independent of $l$. Recall that $C_l$ is the $|.|_{0}$-bound of the Cauchy problem of the $l$th time step. The independence of the substep number $k$ is due to the fact that there is a uniform bound of the first order coefficients for each subproblem of the iteration.The constant $C^*_n$ is determined according to the following lemma (in the remark below we explain in further detail how it is used for convenience of the reader). 
We need the following facts: at each time step $l$ we solve the multivariate Burgers equation on the domain $\left[T_{l-1},T_l\right]\times {\mathbb R}^n$ (original domain) . The final data of the time step $l$ are  the initial data of the time step $l+1$ {\it and} figure as the first order coefficient functions of the first substep problem at time step $l+1$. For this reason we have to ensure  that the initial data of the next time step are H\"{o}lder continuous (in order to get a classical solution for the linear problem of the first substep of time step $l+1$). Note that we encounter linear parabolic equations in substeps. In the $k$th sunstep of time step $l+1$ the corection terms $\delta u^{\rho ,k,l+1},~k\geq 1$ have a fundamental representation with source terms. In order to estimate these representations we need the source term to be bounded continuous. Since the source terms involve the first derivative of the solution of the last time step $l$, we need bounded continuity of the first derivative to be uniformly controlled (independent of the time step). The following lemma and remark summarize the results needed and how they are applied in our situation. 

\begin{lem}\label{lemma}
Let $T>0$ be some horizon. Consider the equation (recall that $0<\nu<\infty$ is some constant)
\begin{equation}\label{parabolcauchy**}
\frac{\partial u}{\partial t}= \nu\Delta u
+\sum_{i=1}^nb_i(t,x)\frac{\partial u}{\partial x_i}+g
\end{equation}
along with the initial condition
\begin{equation}\label{initialf}
u(0,x)=f(x)
\end{equation}
on the domain $D:=[0,T]\times {\mathbb R}^n$. Then we have:
\begin{itemize}
 \item[(i)] Assume that $u$ is bounded and continuous, that the time derivative $u_t$ exists for any $t\in (0,T)$, and that for any $t\in [0,T]$ the spatial derivatives $u_{x_i}$ and $u_{x_ix_j}$ exist and are continuous. Furthermore assume that the coefficients $b_i$ are bounded on $D$ for $1\leq i\leq n$, and that $u$ satisfies the equations (\ref{parabolcauchy**}) and (\ref{initialf}). Then we have
 \begin{equation}
 |u|_0\leq T|g|_0+|f|_0.
 \end{equation}
 \item[(ii)] Assume that $T=1$. Let $g\equiv 0$, and assume that $|f|_{1}\leq C^1_f$, and $|b|_{0}\leq C^0_b$, $|b|_{0,1}\leq C^1_b$ holds for finite constants $C^1_f, C^0_b,C^1_b$.  Then the Cauchy problem (\ref{parabolcauchy**}), (\ref{initialf}) has a unique classical solution $u\in C^{1,2}_b$, and for some constant $C^*_1>0$ (which depends on dimension $n$ and on $\nu$ only) we have
 $$
 |u|_{0,1}\leq \left(1+C^0_bC^*_1\right)|f|_1 
 $$
 \item[(iii)] Let $g\equiv 0$, and assume that $|f|_{2}\leq C_f$, and $|b|_{0,1}\leq C^1_b$ $|b|_{1,2}\leq C^2_b$ holds for finite constants $C_f,C^1_b,C^2_b$.  Then the Cauchy problem (\ref{parabolcauchy**}), (\ref{initialf}) has a unique classical solution $u\in C^{1,2}_b$, and for some constant $C^*_2>0$ (which depends on dimension $n$ and on $\nu$ only) we have
 $$
 |u|_{1,2}\leq \left(1+C^1_b C^*_2\right) |f|_2 
 $$
 \item[(iv)] Assume that $T=1$. Let $f\equiv 0$, and assume that 
 $|g|_{\alpha/2,\alpha}\leq C_g$, and $|b|_{1,2}\leq C_b$ holds for finite constants $C_g,C_b$.
  Then the Cauchy problem (\ref{parabolcauchy**}), (\ref{initialf}) has a unique classical solution $u\in C^{1,2}_b$, and for some constant $C^{**}>0$ (which depends on dimension $n$ and on $\nu$, and on the H\"{o}lder constant $C_g$ only) we have
 $$
 |u|_{1,2}\leq C^{**}.
 $$ 
\end{itemize}
\end{lem}

\begin{proof}
Assertion (i) follows from the maximum principle. Next we prove assertion (ii). Since $\nu>0$ is constant the $b_i$ are H\"{o}lder continuous with respect to the spatial variables uniformly in $t$ the fundamental solution $p$ 
of (\ref{parabolcauchy**}) exists. Since $f\in C^2_b$ classical theory tells us  that a unique classical solution $u$ exists on $[0,T]\times {\mathbb R}^n$, and  (in case $g\equiv 0$) $u$ has the representation
\begin{equation}
u(t,x)=\int_{{\mathbb R}^n}f(y)p(t,x;0,y)dy.
\end{equation}
Recall that the fundamental solution $p$ in the Levy expansion form is given by
\begin{equation}
p(t,x;s,y):=N(t,x;s,y)+\int_s^t\int_{{\mathbb R}^n}N(t,x;\sigma,\xi)\phi(\sigma,\xi;s,y)d\sigma d\xi,
\end{equation}
where
\begin{equation}
N(t,x;s,y)=\frac{1}{\sqrt{4\pi \nu t}^n}\exp\left(-\frac{|x-y|^2}{4\nu t} \right),
\end{equation}
and $\phi$ is a recursively defined function which is H\"{o}lder continous in $x$, i.e.,
\begin{equation}
\phi(t,x;s,y)=\sum_{m=1}^{\infty}(LN)_m(t,x;s,y),
\end{equation}
along with the recursion
\begin{equation}
\begin{array}{ll}
(LN)_1(t,x;s,y)=LN(t,x;s,y)\\
\\
(LN)_{m+1}:=\int_s^t\int_{\Omega}LN(t,x;\sigma,\xi)\left( LN\right)_m(\sigma,\xi;s,y)d\sigma d\xi.
\end{array}
\end{equation}
Now, since the time derivative applied to $N$ cancels out with the Laplacian applied to $N$ we have
\begin{equation}
LN=\sum_{j=1}^nb_j\frac{\partial N}{\partial x_j}.
\end{equation}
 Since $\nu$ is a constant the $N$ depends on $x-y$ and $t-s$. Therefore we write $N^*(t-s,x-y):=N(t,x;s,y)$.
 Hence, for the spatial derivatives of first order we get
\begin{equation}\label{firstder}
\begin{array}{ll}
\int_{{\mathbb R}^n}f(y)p_{x_j}(t,x;0,y)dy=\int_{{\mathbb R}^n}f(y)N^*_{x_j}(t,x-y;0)dy\\
\\
+\int_{{\mathbb R}^n}\int_0^t\int_{{\mathbb R}^n}N^*_{x_j}(t,x-\xi;\sigma)\phi(\sigma,\xi;0,y)f(y)d\sigma d\xi dy\\
\\
=-\int_{{\mathbb R}^n}f_{x_j}(y)N^*(t,x-y;0)dy\\
\\
-\int_{{\mathbb R}^n}\int_0^t\int_{{\mathbb R}^n}N^*_{x_j}(t,x-\xi;\sigma)\Phi(\sigma,\xi;0,y)f_{x_j}(y)d\sigma d\xi dy
\end{array}
\end{equation} 
by partial integration and change of order of integration.
Here 
\begin{equation}
\Phi(t,x;s,y)=\sum_{m=1}^{\infty}(KN^*)_m(t,x;s,y),
\end{equation}
with 
\begin{equation}
\begin{array}{ll}
(KN^*)_1(t,x;s,y)=KN(t,x;s,y)\\
\\
(KN^*)_{m+1}:=\int_s^t\int_{\Omega}\left( LN^*\right)_m(t,x;\sigma,\xi)\left( KN^*\right)_1(\sigma,\xi;s,y)d\sigma d\xi.
\end{array}
\end{equation}
Here,
\begin{equation}
KN(t,x;s,y)=\sum_{j=1}^nb_j(t,y)N^*(t-s,x-y).
\end{equation}
and $\left( LN^*\right)_m$ is defined as above.
The estimate of part (ii) of lemma \ref{lemma} then follows from (\ref{firstder}) and 
the standard estimates
\begin{equation}
\begin{array}{ll}
N^*(t-s,x-y)\leq \frac{C}{t^{\mu_0}\left(x-y\right)^{n-2\mu_0}},\\
\\
N^*_{x_j}(t-s,x-y)\leq \frac{C}{t^{\mu_1}\left(x-y\right)^{n+1-2\mu_1}},
\end{array}
\end{equation}
where $0<\mu_0<1$, and $1/2<\mu_1<1$.
Similarly, for the spatial derivatives of second order we get
\begin{equation}\label{secondder}
\begin{array}{ll}
\int_{{\mathbb R}^n}f_{x_jx_m}(y)N^*(t,x-y;0)dy\\
\\
+\int_{{\mathbb R}^n}\int_0^t\int_{{\mathbb R}^n}N^*_{x_j}(t,x-\xi;\sigma)\Phi^*_0(\sigma,\xi;0,y)f_{x_j}(y)d\sigma d\xi dy\\
\\
+\int_{{\mathbb R}^n}\int_0^t\int_{{\mathbb R}^n}N^*_{x_j}(t,x-\xi;\sigma)\Phi^*_1(\sigma,\xi;0,y)f_{x_jx_m}(y)d\sigma d\xi dy.
\end{array}
\end{equation} 
by partial integration and change of order of integration.
Here 
\begin{equation}
\Phi^*_0(t,x;s,y)=\sum_{m=1}^{\infty}(K^*N^*)_m(t,x;s,y),
\end{equation} 
 \begin{equation}
\begin{array}{ll}
(K^*N^*)_1(t,x;s,y)=&K^*_0N(t,x;s,y)\\
\\
(K^*N^*)_{m+1}(t,x;s,y):=&\int_s^t\int_{\Omega}K^*_1(t,x;\rho,\eta)\left( LN^*\right)_m(\rho,\eta,x;\sigma,\xi)\times \\
\\
&\left( K^*_0N^*\right)_1(\sigma,\xi;s,y)d\sigma d\xi d\eta d\rho,
\end{array}
\end{equation}
where 
\begin{equation}
K^*_1(t,x;\rho,\eta)=\sum_i\left( \frac{\partial}{\partial x_j}b_i(\rho,\eta)\right) N(t-\rho,x-\eta),
\end{equation}
and
\begin{equation}
\Phi^*_1=2\Phi.
\end{equation}
It is worthwile to consider the lowest order term of the latter Levy-type expansion. Iterated partial integration leads to the representation
\begin{equation}\label{secondderlow}
\begin{array}{ll}
u(t,x)_{x_lx_m}(t,x)=\int_{{\mathbb R}^n}f_{x_jx_m}(y)N^*(t,x-y;0)dy\\
\\
+\int_{{\mathbb R}^n}\int_0^t\int_{{\mathbb R}^n}N^*_{x_l}(t,x-z;\sigma){\Big (}\left(\sum_j \left( \frac{\partial}{\partial x_m}b_j(t,x)N(\sigma ,z-y)\right)f_{x_j}(y)\right) +\\
\\
\left(2\sum_j \left( b_j(t,x)N(\sigma,z-y)\right)\right) f_{x_jx_m}(y){\Big)} d\sigma d\xi dy\\
\\
+\mbox{ terms with weaker singularity integrands }
\end{array}
\end{equation} 
Note that representations of the first order time derivative can be expressed in terms of spatial derivative up to second order. Hence
the estimate in part (iii) of lemma \ref{lemma} follows from the representation in equation (\ref{secondder}) with lowest order term of form (\ref{secondderlow}). Finally the estimate of part (iv) is standard (cf. \cite{F2}).

\end{proof}

\begin{rem}\label{remark*}
At each time step the lemma \ref{lemma} is applied with $T=1$ in transformed $\tau$ coordinates. 
Assume that $l$ time steps have been performed leading to a solution $u^{\rho,l}$ which is defined on the time interval $[0,l]$ with respect to transformed coordinates $\tau$. Then part (i) of lemma \ref{lemma} (i), i.e. the maximum priciple is applied in order to show that the supremum norm of the function $u^{\rho,1,l+1}$ is bounded by the supremum of the initial data of the $(l+1)th$ time step, i.e. the quantity $|u^{\rho,l}(l,.)|_0$. Then part (i) of lemma \ref{lemma} with zero initial data and recursively defined right side is applied recursively. The choice of the time step size $\rho_l$ ensures convergence of the functional series with the elements $\delta u^{\rho,k,l+1}_i$ for $k\geq 2$ with respect to the supremum norm. In order to construct a classical solution we need estimates of the first and second spatial derivatives and of the first time derivatives. In order to estimate  $|u^{\rho,1,l+1}|_{0,1}$ part (ii) of lemma \ref{lemma} is applied, and in order to estimate $|u^{\rho,1,l+1}|_{1,2}$ part (iii) of lemma \ref{lemma} is applied. From the proof of part (ii) of lemma \ref{lemma} we see that the quantity  $|u^{\rho,1,l+1}|_{0,1}$ is bounded by the the quantity $|u^{\rho,l}(l,.)|_1$ plus a series involving first order coefficients $b_i$ and multiple integrals of first order coefficients $b_i$ times the quantity $|u^{\rho,l}(l,.)|_1$. Since the first order coefficients have a multiplier $\rho_l$ from the representation in the proof of lemma \ref{lemma} we observe that the choice of $\rho_l$ of form $\frac{1}{4C^*_nC_l}$ with some $C^*_n$ depending only on $\nu$ and the dimension $n$ (which makes the first order coefficient of the Cauchy problem for $u^{\rho,1,l+1}$ small) ensures that we have the bound $|u^{\rho,1,l+1}|_{0,1}\leq |u^{\rho,l}(l,.)|_1\left(1+C\frac{\rho_l}{1-\rho_lC^0_b} \right) $ for some constant $C>0$. Note that powers of $\rho_l$ come from the representation of the higher order terms in Levy expansion and are estimated here in form of a geometric series.  Similarly, from the proof of part (iii) of lemma \ref{lemma} we see that the quantity  $|u^{\rho,1,l+1}|_{1,2}$ is bounded by the the quantity $|u^{\rho,l}(l,.)|_2$ plus a series involving first order coefficients $b_i$ and first order derivatives of $b_i$ and multiple integrals of first order coefficients $b_i$ and its first order derivatives times the quantity $|u^{\rho,l}(l,.)|_2$. This leads to a bound $|u^{\rho,1,l+1}|_{1,2}\leq |u^{\rho,l}(l,.)|_2\left(1+\tilde{C}\frac{\rho_l}{1-\rho_lC^1_b}\right) $ for some constants $C_b$ and $\tilde{C}>0$. Next part (iv) of lemma \ref{lemma} is used in order to estimate the 'correction terms' $\delta u_i^{\rho,k,l+1}$ with respct to the norm $|.|_{0,1}$ and $|.|_{1,2}$. Note that we may choose $C^*_n$ dependent only on $\nu$ and the dimension $n$ large enough such that the bounds are independent of the time step $l$.
\end{rem}

\section{Further remarks on uniqueness and regularity}

Uniqueness is a consequence of a global existence result for certain semilinear systems (cf. also proposition 8.6 below).
\begin{thm}\label{unique}
Let $\Omega={\mathbb R}^n$ or $\Omega={\mathbb T}^n$. There is a unique globally bounded and H\"older continuous classical solution of the Cauchy problem  (\ref{qparasyst1}) of the multivariate Burgers equation.
\end{thm}

\begin{proof}
 Let ${\bf v}^1$ and ${\bf v}^2$ be two globally bounded and H\"older continuous solutions of the Cauchy problem (\ref{qparasyst1}). The difference ${\bf \delta v}:={\bf v}^1-{\bf v}^2$ with components $\delta v_i:= v^1_i-v^2_i$  satisfies
\begin{equation}\label{unique}
\left\lbrace \begin{array}{ll}
\frac{\partial \delta v_i}{\partial t}-\sum_{j=1}^n \frac{\partial^2 \delta v_i}{\partial x_j^2} 
+\sum_{j=1}^n v^1_j\frac{\partial \delta v_i}{\partial x_j}+\sum_{j=1}^n \delta v_j\frac{\partial  v^2_i}{\partial x_j}=0,\\
\\
\mathbf{\delta v}={\bf 0},
\end{array}\right.
\end{equation}
Since $v_i^1$ and $v_i^2$ are given and H\"older continuous, we have a semilinear system which is a special case of the linear system (2.8) described in (\cite{BS}). Since $v_i^1,v_i^2$ are globally bounded and H\"older continuous the assumptions (2.9)-(2.13) of proposition 2.3 in (\cite{BS}) are satisfied in case $\Omega ={\mathbb R}^n$ and according to proposition (2.3) there is a unique global classical solution $\delta v\in C^{1,2}_b\left( [0,T)\times {\mathbb R}^n,{\mathbb R}^n\right) $. Since $\mathbf{\delta v}={\bf 0}$ is a solution of (\ref{unique}), we have ${\bf v}^1={\bf v}^2$. The result can be transferred to the case $\Omega={\mathbb T}^n$ easily. 
\end{proof}

\end{document}